\documentclass[11pt,oneside]{amsart}

\usepackage{amsmath,ifthen, amsfonts, amssymb,
srcltx, amsopn, color, enumerate, mathrsfs, overpic}
\usepackage[cmtip,arrow]{xy}
\usepackage{pb-diagram, pb-xy}
\usepackage{marginnote}
\usepackage{overpic}
\usepackage{subfig}
\usepackage{mathtools}

\usepackage[T1]{fontenc}
\usepackage{yfonts}

\usepackage{lineno}
\usepackage{graphicx}

\newcommand{\showcomments}{yes}
\renewcommand{\showcomments}{no}

\newsavebox{\commentbox}
{\ifthenelse{\equal{\showcomments}{yes}}%
{\footnotemark
\begin{lrbox}{\commentbox}
\begin{minipage}[t]{1.25in}\raggedright\sffamily\tiny
\footnotemark[\arabic{footnote}]}
{\begin{lrbox}{\commentbox}}}%
{\ifthenelse{\equal{\showcomments}{yes}}%
{\end{minipage}\end{lrbox}\marginpar{\usebox{\commentbox}}}
{\end{lrbox}}}

\newtheorem{thm}{Theorem}[section]

\newtheorem{lem}[thm]{Lemma}

\newtheorem{cor}[thm]{Corollary}

\newtheorem{prop}[thm]{Proposition}

\theoremstyle{definition}

\newtheorem{claim*}{Claim}

\DeclareMathOperator{\GL}{GL}
\newcommand{\Z}{\ensuremath{\mathbb{Z}}}
\newcommand{\p}{\ensuremath{\mathbb{P}}}
\newcommand{\E}{\ensuremath{\mathbb{E}}}

\makeatletter

\newcommand{\Rmnum}[1]{\mathbf{{\expandafter\@slowromancap\romannumeral #1@}}}

\DeclareMathOperator{\Ave}{Ave}
\makeatother

\let\oldmarginpar\marginpar
\renewcommand\marginpar[1]{\-\oldmarginpar[\raggedleft\footnotesize #1]%
{\raggedright\footnotesize #1}}

\newcounter{enumitemp}

\newcommand{\N}{\ensuremath{\mathbb{N}}}

\newcommand{\ZZ}{\mathbb Z}
\newcommand{\NN}{\mathbb N}
\newcommand{\PP}{\mathbb P}
\newcommand{\EE}{\mathbb E}
\newcommand{\Cay}{{\rm Cay}}

\begin{document}
\title{Expected depth of random walks on groups}
\author[K. Bou-Rabee]{Khalid bou-Rabee}
\thanks{K.B. supported in part by NSF grant DMS-1405609, A.M. supported by Swiss NSF grant P2GEP2\_162064.}
\address{School of Mathematics, CCNY CUNY, New York City, New York, USA}
\email{khalid.math@gmail.com}
\author[I. Manolescu]{Ioan Manolescu}
\thanks{I.M. affiliated to NCCR SwissMAP}
\address{Universit\'e de Fribourg, Fribourg, Switzerland}
\email{ioan.manolescu@unifr.ch}
\author[A. Myropolska]{Aglaia Myropolska}
\address{Laboratoire de Math\'ematiques,
Universit\'e Paris-Sud 11, Orsay, France}
\email{aglaia.myropolska@gmail.com}

\date{\today}
\maketitle

\begin{abstract}
	For $G$ a finitely generated group and $g \in G$, we say $g$ is detected by a normal subgroup $N \lhd G$ if $g \notin N$. 
	The depth $D_G(g)$ of $g$ is the lowest index of a normal, finite index subgroup $N$ that detects $g$.
	In this paper we study the expected depth, $\EE[D_G(X_n)]$, where $X_n$ is a random walk on $G$. 
	We give several criteria that imply that 
	$$\EE[D_G(X_n)] \xrightarrow[n\to \infty]{} 2 + \sum_{k \geq 2}\frac{1}{[G:\Lambda_k]}\, ,$$
	where $\Lambda_k$ is the intersection of all normal subgroups of index at most $k$.  In particular, the equality holds in the class of all nilpotent groups and in the class of all linear groups satisfying Kazhdan Property $(T)$.
	We explain how the right-hand side above appears as a natural limit and also give an example where the convergence does not hold. 
\end{abstract}

\section{Introduction}

%Let $G$ be a finitely generated, infinite group. 
%An element $g \in G$ is said to be detected by a normal subgroup $N \lhd G$ if $g \notin N$. 
%We will always consider this notion for normal subgroups $N$ of finite index. 

%Define the  \emph{depth} of an element $g \in G \setminus \{e\}$ as 
%$$D_G(g)=D_G(g)=\min\{|G/N|: N\lhd_{\text{finite index}} G \text{ and } g\notin N\}.
%$$

Let $G$ be a finitely generated group.
The \emph{depth} of an element in $G$ encodes how well approximated that element is by finite quotients of the group.
The goal of this article is to find the average depth of an element of $G$.
As such, this question is ill-posed, and a more precise one is:
\emph{what is the asymptotic expected depth of a random walk on the Cayley graph of $G$?}
This question arises naturally when quantifying \emph{residual finiteness}, 
or in other words, when studying statistics surrounding the \emph{depth function}.

For $g \in G$ and $N$ a normal subgroup of $G$, we say $g$ is \emph{detected by} $N$ if $g \notin N$ 
(in other words, if $g$ is mapped onto a non-trivial element of $G/N$). 
The \emph{depth} of $g$ is the lowest index of a normal, finite index subgroup $N$ that detects $g$. 
Formally, for $g\in G$, $g \neq e$, set
$$D_G (g) := \min\{|G/N|: N\lhd_{\text{finite index}} G \text{ and } g\notin N\}.$$
For $g = e$ the above definition would produce a depth equal to $\min \emptyset = \infty$. 
In the context of random walks, this singularity would produce trivial results. 
To circumvent this triviality, we instead define $ D_G(e) := 0.$
With this definition, $G$ is {\em residually finite} if and only if $D_G(g) < \infty$ for all elements $g \in G$. 

Let $G$ be residually finite and $S$ be a finite generating set, which will be always considered symmetric. 
Then the \emph{residual finiteness growth function} is
$$F_G^S(n)=\max_{g\in B_G^S(n)} D_G(g),$$
where $B_G^S(n)$ is the ball of radius $n$ in the Cayley graph $\Cay(G,S)$.
This notion was introduced in \cite{B09} and has been studied for various classes of groups; 
the relevant results to this paper are listed in \S\ref{depth}.

While the residual finiteness growth function reflects the largest depth of an element in the ball of radius $n$, 
the question of interest in this paper is what we can say about the ``average'' depth of a ``uniform'' element of the group. 
If $G$ is discrete but infinite, there is no natural definition of a uniform probability measure on $G$, hence no good notion of a uniform element. 
However, we may try to approach the desired ``average'' by averages of well defined measures. 
Two approaches come to mind: 
\begin{itemize}
	\item For $n \geq 0$, let $Z_n$ be a uniform element in $B_G^S(n)$, and let 
	$$a_n = \mathbb E[D_G(Z_n)] = \frac{1}{|B_G^S(n)|}\sum_{g \in B_G^S(n)} D_G(g).$$
	We could then say that the average depth of an element of $G$ is $\lim_n a_n$, provided that this limit exists. 
	\label{a}
	\item Alternatively, one may define a random walk $(X_n)_{n \in \mathbb N}$ on a Cayley graph $Cay(G,S)$ of $G$, 
	starting from the neutral element $e$, and set 
	$$ b_n =  \mathbb E[D_G(X_n)].$$
	Then define the average depth as $\lim_n b_n$, again under the condition that the limit exist.  
\end{itemize}
One expects that for compliant groups, both limits exist and are equal. 
We will focus on the second situation; but will make reference to the first to stress similarities. 
The exact definition of $(X_n)_{n \geq 0}$ as well as a discussion on random walks on groups 
is deferred to \S\ref{sec:RW}. 
We mention here only that $(X_n)_{n\geq 0}$ is a {\em lazy} random walk, 
that is a process that at every step remains unchanged with probability $1/2$ and takes a step otherwise. 

%In \cite{Khalid's paper}, a tentative formula for the limit in~\ref{a}  was given. We explain this next. 

It is a general fact that for an integer valued non-negative random variable $Y$, 
$$\mathbb E(Y) = \sum_{k \geq 0} \mathbb P (Y > k).$$
It may therefore be interesting to study $\mathbb P[D_G(X_n) > k]$ for $(X_n)_{n\geq 0}$ as above. 

For $k \geq 2$, let $\Lambda_k$ be the intersection of all normal subgroups of $G$ of index at most $k$. 
(For $k =0,1$, set $\Lambda_k = G$). 
Then, for $g \in G \setminus \{e\}$, $D_G(g) > k$ if and only if $g \in \Lambda_k$. 
Thus 
\begin{align}\label{eq:Esum}
	\mathbb E[D_G(X_n)] = \sum_{k \geq 0} \mathbb P (X_n \in \Lambda_k\setminus\{e\}).
\end{align}
As we will see in Corollary~\ref{cor:uniform}, 
$$ \mathbb P (X_n \in \Lambda_k\setminus\{e\}) \xrightarrow[n \to \infty]{} \frac{1}{[G:\Lambda_k]}.$$
One may therefore expect that 
\begin{align}\label{eq:goal}
	\mathbb  E[D_G(X_n)] \xrightarrow[n \to \infty]{} 2 + \sum_{k \geq 2}\frac{1}{[G:\Lambda_k]},
\end{align}
where the factor $2$ appears since $[G:\Lambda_1]=[G:\Lambda_0]=[G:G]=1$.
For this reason, we call the right hand side of the above the \emph{presumed limit}.
However, the convergence above is far from obvious.
The main goal of this paper is to provide criteria for $G$ under which~\eqref{eq:goal} holds. 
We will also provide an example where this is not valid. 

The finiteness of  $\sum_{k\geq 2}\frac{1}{|G:\Lambda_k|}$ in~\eqref{eq:goal} depends on the group $G$ and is related to the \emph{intersection growth} $i_G (k)=[G:\Lambda_k]$ of $G$. 
It follows from Lemma~\ref{linear-lemma} and Theorem~\ref{thm:linear-finite average} that $\sum_{k\geq 2}\frac{1}{|G:\Lambda_k|} < \infty$ for finitely generated linear groups.
Moreover, finitely generated nilpotent groups enjoy this property as it is a classical result that they are linear (see Segal \cite[Chapter 5, \S B, Theorem 2]{MR713786} or Hall \cite[p. 56, Theorem 7.5]{MR0283083}).
%Indeed, by the Malcev-Hall Embedding Theorem (see Segal \cite[Chapter 5, \S B, Theorem 2]{MR713786} or Hall \cite[p. 56, Theorem 7.5]{MR0283083}) any torsion free finitely generated nilpotent group is a linear group.
%Since any finitely generated nilpotent group contains a subgroup of finite index that is torsion free, linearity follows by an induced representation argument.

Our two results ensuring~\eqref{eq:goal} are the following.

\begin{thm}
	Let $G$ be a linear group with Kazhdan Property $(T)$. Then 
	$$\lim_{n\rightarrow \infty}\mathbb{E}[D_G(X_n)] = 2 +\sum_{k \geq 2} \frac1{[G:\Lambda_k]}<\infty$$
	for any finite generating set $S$ of $G$.
	In particular,  $\lim_{n\rightarrow \infty}\mathbb{E}[D_G(X_n)]$ is finite for the special linear groups $SL_k(\mathbb{Z})$ with $k\geq 3$.
	\label{kazhdan}
\end{thm}

\begin{thm}
	Let $G$ be a finitely generated nilpotent group. Then 
	$$\lim_{n\rightarrow \infty}\mathbb{E}[D_G(X_n)] = 2 +\sum_{k \geq 2} \frac1{[G:\Lambda_k]}<\infty$$
	for any finite generating set $S$ of $G$.
	\label{nilp}
\end{thm}

In \S\ref{sec:Fatou}, it will also be shown that the convergence holds whenever the presumed limit is infinite. 
Considering these examples, one may think that the convergence in~\eqref{eq:goal} is always valid. 
However, in Proposition~\ref{prop:counter_ex}, we exhibit a $3$-generated group for which the presumed limit is finite but 
$\lim_{n\rightarrow \infty}\mathbb{E}[D_G(X_n)] = \infty$.

Henceforth, when no ambiguity is possible, we drop the index $G$ from the notation $D_G(.)$. 

\section{Preliminaries}

\subsection{Depth function and residual finiteness growth}\label{depth}

This short subsection includes some results on the residual finiteness growth function that we will use in the sequel. 

%First note, that the asymptotic behaviour of $F_G^S$ does not depend on the choice of the generating set, see \cite[Lemma 1.1]{B09}. Namely, for two generating sets $S_1$ and $S_2$ of $G$ the functions $F_G^{S_1}$ and $F_G^{S_2}$ are equivalent.
%Namely, we say that $f\preseq g$ if there exist positive constants $C_1, C_2$ such that $f(n)\leq C_1 g(C_2)
% Recall that two functions $f,g: \mathbb{R}\rightarrow \mathbb{R}$ are equivalent $f\approx g$ if there exist positive constants $c_1, c_2, c_3, c_4$ such that $f(x)\leq c_1g(c_2x)$ and $g(x)\leq c_3f(c_4x)$ for all $x\in \mathbb{R}$.

\begin{thm}[\cite{B09}]
	Let $G$ be a finitely generated nilpotent group with a generating set $S$. Then 
	\begin{align*}
	F_G^S(n)\leq C\log (n)^{h(G)}, \qquad \forall n \geq 2,
	\end{align*}
	where $h(G)$ is the Hirsch length of $G$ and $C=C(G,S)$ is a constant independent of $n$.
	\label{rf-nilp}
\end{thm}

The Prime Number Theorem and Hall's Embedding Theorem play key roles in the proof of Theorem~\ref{rf-nilp}. In \cite{BM13}, the following is proved using Gauss's Counting Lemma to help quantify Mal'cev's classical proof of residual finiteness of finitely generated linear groups.

\begin{thm}[\cite{BM13}]
	Let $K$ be a field.
	Let $G$ be a finitely generated subgroup of $GL(m,K)$ with a generating set $S$. 
	Then there exists a positive integer $b$ such that 
	\begin{align*}
		F_G^S(n)\leq C n^b,\qquad \forall n \geq 1,
	\end{align*}
	where C=C(G,S) is a constant independent of $n$.
\end{thm}

The above results bound from above the residual finiteness growth. 
Conversely, the following states that there exist groups with arbitrary large residual finiteness growth.  

\begin{thm}[\cite{BS13b}] \label{rf-neumann}
	For any function $f\colon \N\rightarrow \N$, 
	there exists a residually finite group $G$ and a two element generating set $S$ for $G$,
	such that $F_{G}^S(n) \geq f(n)$ for all $n\geq 8$.
\end{thm}

\noindent
The proof of Theorem~\ref{rf-neumann} in \cite{BS13b} involves an explicit construction of a finitely generated group embedded in an infinite product of finite simple groups.

%We also mention that for some branch groups, including the family of Grigorchuk groups and Gupta-Sidki groups, the residual finiteness growth is nearly exponential \cite{}.

\subsection{Random walks on groups}\label{sec:RW}

Let $G$ be a finitely generated group with a finite symmetric generating set $S=\{s_1, \dots, s_k\}$, \emph{i.e.} such that $S^{-1}=S$. A random walk $(X_n)_{n\geq 0}$ on $G$ is a Markov chain with state space $G$ and such that $X_0=e_G$ and $X_{n+1}=X_n\cdot Y_n$ for $n\geq 0$ where $Y_0, Y_1, \dots$ are independent and uniform in $\{s_1,\dots, s_k\}$.
%\im{Talk a bit about RW and lazy random walk. Say at some point that from now on we consider only lazy RW}

If $G$ is finite with $|G|=m$, one may consider the transition matrix $P$ of the random walk  $(X_n)_{n\geq 0}$ on $G$ defined by
\begin{align*}
	P(x,y) =\frac{1}{|S|}\sum_{s \in S} {\mathbf 1}_{\{y = x s\}},
\end{align*}
where $\mathbf 1_{\{y = x s\}} = 1$ if $y = xs$ and $0$ otherwise. 
It is simply the adjacency matrix of the Cayley graph $\Cay(G,S)$, normalized by $\frac{1}{|S|}$. 
The generating set is considered symmetric so as to have an un-oriented Cayley graph, or equivalently to have $P$ symmetric. 

 Let $1=\lambda_1 \geq \dots \geq \lambda_m \geq -1$ be the eigenvalues of $P$ and $x_1, \dots, x_m$ be a basis of orthonormal eigenvectors of $P$ (such a basis necessarily exists since $P$ is real and symmetric). Let $\sigma$ be an initial distribution on $G$ seen as a probability vector of dimension $m$, and let $p_u=(\frac{1}{m},\dots, \frac{1}{m})$ be the uniform distribution on $G$.
It is well-known that the distribution of such random walk converges to the uniform distribution whenever the graph is assumed to not be bipartite.
For a general convergence statement one considers a \emph{lazy} random walk instead; that is a walk with transition matrix $L=\frac{1}{2} I+ \frac{1}{2} P$. The lazy random walk at time $n$ takes a step of the original random walk with probability $\frac{1}{2}$ and stays at the current vertex with probability $\frac{1}{2}$.
Notice that the eigenvectors of $L$ are $x_1, \dots, x_m$ and the corresponding eigenvalues are all non-negative $\mu_1=\frac{1}{2}+\frac{1}{2}\lambda_1=1 > \mu_2=\frac{1}{2}+\frac{1}{2}\lambda_2\geq \dots \geq \mu_m=\frac{1}{2}+\frac{1}{2}\lambda_m \geq 0$.

\begin{lem} 
	Let $G$ be a finite group with a finite symmetric generating set $S$. 
	With the above notation $||\sigma L^n-p_u||_2 \leq \mu_2^n$. In particular, $|\sigma L^n (g) - \frac{1}{m}|\leq \mu_2^n$ for every $g\in G$.
	\label{finite-unif}
\end{lem}
\begin{proof}
For $n\geq 1$, the matrix $\sigma \cdot L^n$ is a probability distribution and it represents the distribution of the $n$-the step of the lazy random walk on $G$ that starts at a random vertex selected according to $\sigma$.
%For $n\geq 1$, the matrix $\sigma \cdot P^n$ is a probability distribution and it represents the distribution of the final vertex in a $n$-the step of the random walk on $G$ that starts at a vertex selected according to $\sigma$.
%The eigenvalues and the eigenvectors of $P^n$ are respectively $1=\lambda_1^n\geq \lambda_2^n\geq \dots \geq \lambda_m^n$ and $x_1, x_2, \dots, x_m$.

We write $\sigma=\alpha_1 x_1+ \alpha_2 x_2+\dots+\alpha_m x_m$, with $\alpha_1,\dots,\alpha_m \in \mathbb{R}$. Since $x_1,\dots, x_m$ are eigenvectors, we have $\sigma L^n= \alpha_1\mu_1^n x_1+\alpha_2 \mu_2^n x_2+\dots + \alpha_m\mu_m^n x_m$.
Notice that $\mu_1=1$, $x_1=\frac{1}{\sqrt{m}}(1, \dots, 1)$ and $\alpha_1=\sigma \cdot x_1^T=\frac{1}{\sqrt{m}}$ which implies that $\alpha_1 x_1=(\frac{1}{m}, \dots, \frac{1}{m}) = p_u$.
We deduce that
$$
	\big\|\sigma L^n-p_u\big\|_2=\big\|\alpha_2 \mu_2^n x_2+\dots + \alpha_m\mu_m^n x_m\big\|_2
	\leq \max_{i=2,\dots, m} |\mu_i|^n\cdot \sqrt{\alpha_2^2+\dots + \alpha_m^2}
	\leq \mu_2^n \cdot \|\sigma\|_2.
$$
In the last line we used the orthonormality of the base $x_1,\dots, x_m$. Finally, $\|\sigma\|_2 \leq \sum_{i=1}^m\sigma_i = 1$, which shows that the above is bounded by $\mu_2^n $ as required. 
\end{proof}

\begin{cor}	\label{cor:uniform}
	Let $G$ be a finitely generated infinite group with a finite symmetric generating set $S$ 
	and let $N$ be a normal subgroup of $G$ of finite index. 
	Consider the lazy random walk $(X_n)_{n\geq0}$ on $\Cay(G,S)$. Then
	\begin{align}\label{eq:uniform}
		\Big|\mathbb{P}(X_n\in N)-\frac{1}{|G:N|}\Big| \leq \mu^n,
	\end{align}
	where $\mu$ is the second largest eigenvalue of the transition matrix of $\tilde{X_n}$, 
	the lazy random walk on $\Cay(G/N, S)$ induced by  $(X_n)_{n\geq 0}$.
	Moreover, 
	\begin{align}\label{eq:uniform2}
		\mathbb{P}(X_n\in N \setminus \{e\})\xrightarrow[n\to\infty]{}\frac{1}{|G:N|}.
	\end{align}
\end{cor}

\begin{proof}
	%The walk $X_n$ induces a lazy random walk $\tilde{X_n}$ on the finite quotient $G/N$.
	To prove~\eqref{eq:uniform} it suffices to observe that $\mathbb{P}(X_n\in N)=\mathbb{P}(\tilde{X_n} = e_N)$ 
	and conclude by Lemma~\ref{finite-unif}.
	Let us now show~\eqref{eq:uniform2}.
	
	It is a standard fact (see for instance \cite[Thm. VI.3.3, Thm. VI.5.1]{Varopoulos_book}) that, since $G$ is infinite,  
	$$ \mathbb{P}(X_n = e) \to 0, \qquad \text{as $n\to \infty$}. $$
	Moreover, as discussed above, the eigenvalue $\mu$ appearing in~\eqref{eq:uniform} is strictly smaller than $1$. 
	These two facts, together with~\eqref{eq:uniform}, imply~\eqref{eq:uniform2}.
\end{proof}

In the rest of the paper, we will always consider lazy random walks as described above. 
Straightforward generalisations are possible, such as to random walks with non-uniform symmetric transition probabilities --
that is, walks taking steps according to a finitely supported, symmetric probability on $G$, 
with $e$ having a positive probability (which is to say that the walk has, at any given step, a positive probability of staying at the same place). 
Certain transition probabilities with infinite support (but finite first moment) may also be treated, 
but small complication arise in specific parts of the proof.
For the sake of legibility, we limit ourselves to the simple framework of uniform probabilities on symmetric generating sets. 

\subsection{Asymptotic density}

For a given infinite group $G$ generated by a finite set $S$, consider the so-called \emph{asymptotic density} (as defined in \cite{BV02}) of a subset $X$ in $G$ defined as follows
\begin{align}\label{eq:rho}
	\rho_S(X)=\limsup_{n\rightarrow \infty}\frac{|X\cap B_G^S(n)|}{|B_G^S(n)|}.
\end{align}

If $G$ satisfies 
\begin{equation} 
	\lim_{n\rightarrow\infty}\frac{|B_G^S(n+1)|}{|B_G^S(n)|}=1
\label{cond1}
\end{equation}
then by \cite{BV02} $\rho_S$ is left- and right-invariant; and in particular, $\rho(H)=\frac{1}{|G:H|}$ for a finite index subgroup $H$. 
Moreover, if (\ref{cond1}) holds, the $\limsup$ in~\eqref{eq:rho} is actually a limit. 

Condition~\eqref{cond1} holds for groups of polynomial growth (as a consequence of \cite{Pansu08}).
Thus, for all $k\geq 2$, recalling the random variables $X_n$ and $Z_n$ defined in the introduction,
\begin{align*}
	\lim_{n\to\infty} \mathbb P (X_n \in \Lambda_k\setminus\{e\}) = \lim_{n\to\infty} \mathbb P (Z_n \in \Lambda_k\setminus\{e\}) = \frac{1}{|G:\Lambda_k|}.
\end{align*}

For groups with exponential growth however condition~\eqref{cond1} fails, 
and the second limit in the above display does not necessarily exist.
We give next an example for $G= {\rm F}_{\{a,b\}}$, the free group generated by two elements $\{a,b\}$ 
and for a normal subgroup $N \lhd {\rm F}_{\{a,b\}}$.  
Take $S = \{a,b,a^{-1}, b^{-1}\}$.
For $g \in {\rm F}_{\{a,b\}}$, let $\|g\|$ be the word-length of $g$, 
that is, the graph distance from $g$ to $e$ in $\Cay( {\rm F}_{\{a,b\}}, S)$. 
Set 
$$N = \{g \in {\rm F}_{\{a,b\}} :\, \|g\| \in 2 \NN\}.$$
It is straightforward to check that $N$ is a normal subgroup of ${\rm F}_{\{a,b\}}$ of index $2$. 
However, 
\begin{align*}
	\frac{|N\cap B_{{\rm F}_{\{a,b\}},S}(n)|}{|B_{{\rm F}_{\{a,b\}},S}(n)|} = 
	\begin{cases}
		\displaystyle\frac{3^{n+1} -1}{4 \cdot 3^{n} - 2}, \qquad &\text{ for $n$ even},\medskip \\
		\displaystyle\frac{3^{n} -1}{4 \cdot 3^{n} - 2}, \qquad &\text{ for $n$ odd}.
	\end{cases}
\end{align*}
It is immediate from the above that $\PP(Z_n \in N)$ does not converge when $n\to \infty$. 

The above example, together with Corollary~\ref{cor:uniform}, 
explains the choice of the random walk $(X_n)_{n \geq 0}$ rather than of the uniform variables $(Z_n)_{n \geq 0}$ on $B_G^S(n)$. 

Another reason for this choice relates to sampling. 
Suppose we have sampled an instance of the variable $Z_n$ for some $n \geq 0$. 
In order to then obtain a sample of $Z_{n+1}$, one needs to restart the relatively costly process of sampling a uniform point in $B_G^S(n+1)$. 
For the random walk however, if $X_n$ is simulated for some $n \geq 0$, $X_{n+1}$ 
is easily obtained by multiplying $X_n$ with a random element in $S$. 
This makes the sampling of a sequence $(X_1,X_2,\dots)$ much easier than that of a sequence $(Z_1,Z_2,\dots)$. 

\section{Residual average}

We mentioned in the introduction that our goal is to compute the ``average'' depth of an element in $G$. 
In addition to the two methods proposed above, that is taking the limit of $\EE[D_G(X_n)]$ or $\EE[D_G(Z_n)]$, 
one may compactify $G$ so that it has a Haar probability measure and take the average depth with respect to it. 
The natural way to render $G$ compact is by considering its {\em profinite completion}, which we will denote by $\widehat{G}$. 
It is a compact group, with a unique uniform Haar measure which we denote by $\mu$.  
The depth function $D_G$ may be extended by continuity to $\widehat G$, and we call $D_{\widehat G}$ this extension. 
Then the \emph{residual average} of $G$, denoted by $\Ave(G)$, is
\[\Ave(G) := \int_{\widehat{G}} D_{\widehat{G}}d\mu. \] 
For details of the profinite completion construction see \cite{MR1691054}.
For further details of the residual average construction see \cite{BM10}.

%
%While the residual finiteness growth function gives asymptotic data comparing depth and word length, 
%it is also natural to to study the $L^1$--norm of the depth function. 
%More explicitly, the \emph{residual average}, denoted by $\Ave(G)$ is
%\[ \int_{\widehat{G}} D_{\widehat{G}}d\mu, \]
%where $\widehat{G}$ is the profinite completion of $G$, the measure $\mu$ is the Haar probability measure on $\widehat{G}$, and $D_{\widehat{G}}$ is a continuous extension of $D_G$ to $\widehat{G}$ (see \cite{MR1691054} for the theory of profinite groups).

\begin{lem}
	For any linear group $G$,
	\begin{align}\label{eq:linear-lemma}
		 \Ave(G) = 2 + \sum_{k=2}^{\infty} \frac{1}{|G:\Lambda_k|} .
	\end{align} 
	\label{linear-lemma}
\end{lem}

Note that $\Ave(G)$ in~\eqref{eq:linear-lemma} is equal to the limit in~\eqref{eq:goal}.

\begin{proof}
Recall the fact that we have conveniently defined $\Lambda_0 = \Lambda_1 = G$ and therefore that $\mu(\Lambda_0) = \mu(\Lambda_1) = 1$. 
The residual average is then 
\begin{align*}
	\Ave(G)
	& =\sum_{k=1}^{\infty} k\cdot [\mu(\Lambda_{k-1})-\mu(\Lambda_k)] 
	  = \sum_{ k \geq 1} \sum_{\ell =0}^{k-1} [\mu(\Lambda_{k-1})-\mu(\Lambda_k)]\\
	& = \sum_{\ell \geq 0} \sum_{ k > \ell } [\mu(\Lambda_{k-1})-\mu(\Lambda_k)]\\
	& = \sum_{\ell \geq 0} \mu(\Lambda_{\ell})
%	\\
%	&=& C + \sum_{k=2}^\infty \mu(\Lambda_k) = C+ \sum_{k=2}^\infty \frac{1}{i_G(k)}.
\end{align*}
We are authorized to change the order of summation in the third equality, since all the terms in the sum are non-negative. 
In the last equality, we have used the telescoping sum and the fact that $\mu(\Lambda_k) \to 0$ as $k \to \infty$. 
The latter convergence is due to $G$ being residually finite. 

The first two terms in the last sum above are equal to $1$; for $\ell \geq 2$, $\mu(\Lambda_{\ell}) = \frac{1} {|G:\Lambda_\ell|}$.
The lemma follows immediately.
\end{proof}

The following theorem, taken from \cite{BM10}, will be necessary when proving Theorem~\ref{kazhdan}.

\begin{thm}[Theorem 1.4 \cite{BM10}]\label{thm:linear-finite average}
Let $\Gamma$ be any finitely generated linear group.
Then the residual average of $\Gamma$ is finite.
\end{thm}

For completeness, we give a proof of the above. 
The present proof is based on the one in the original paper, with some adjustments meant to correct certain points. 
The main difference with the original proof is that here we focus on the connection to intersection growth.
One has to be especially careful in proving this result, as the residual finiteness growth may vary when passing to subgroups of finite index (see \cite[Example 2.5]{BK12}).
%We present a simpler proof than the one in \cite{BM10} while correcting some errors. 
% \kb{I am now happy with the maths here. However, I'm still iffy on how we should present this correction in a way that won't anger anyone. Any ideas? What do you think of what's here?} 

\begin{proof}
We follow the proof \cite[Theorem 1.4]{BM10}, with some changes and expansions.
According to \cite[Proposition 5.2]{BM10}, there exists an infinite representation
$$
\rho : \Gamma \to GL(n,K)
$$
for some $n$ and $K/\mathbb{Q}$ finite.
By \cite[Lemma 2.5]{BM10}, it suffices to show that the normal residual average of $\rho(\Gamma)$ is finite. Set $\Lambda = \rho (\Gamma)$ and set $S$ to be the coefficient ring of $\Lambda$.

For each $\delta >0$, from the proof of \cite[Proposition 5.1]{BM10}, there exists a normal residual system $\mathfrak{F}_\delta$ on $\Lambda$ given by $\Delta_{j} = \Lambda \cap \ker r_{j}$, where
$$
r_j  : \GL(n,S) \to \GL(n, S/\mathfrak{p}_{S,j}^{k_j}),
$$
and $[\Lambda:\Delta_j] \leq [\Lambda:\Delta_{j+1}] \leq [\Lambda:\Delta_j]^{1+\delta}$.
In addition, we have
$$
|r_{j}(\Lambda)| = O_j p_j^{\ell_{j}}
$$
where
\begin{equation} \label{Oinequality}
1 \leq O_j < p_j^{n^2}.
\end{equation}
We also have for constants $N > (n^2)!$ and $C > 4$ that
\begin{equation} \label{ellinequality}
\ell_{j} > N + C jn^2
\end{equation}
and
$$
\ell_{j} + Cn^2 < \ell_{j+1} \leq \ell_{j} + (C+1)n^2.
$$

For each $i < j$, we claim that the largest power of $p_j$ that divides $[\Lambda: \Delta_{i}]$ is $p_j^{n^2}$.
To see this claim, note that if $p_j^m$ divides $O_i p_i^{\ell_i}$, since $p_i, p_j$ are distinct primes, $p_j^m$ must divide $O_i$, however $O_i < p_i^{n^2}$ and $p_i < p_j$.
The claim follows.
Since
$$
[\Delta_i : \Delta_i \cap \Delta_j]
= 
[\Lambda : \Delta_j] [\Delta_j : \Delta_i \cap \Delta_j]/ [\Lambda : \Delta_i] 
=
O_j p_j^{\ell_j} [\Delta_j : \Delta_i \cap \Delta_j]/ [\Lambda : \Delta_i] ,
$$
the aforementioned claim implies that
$$
[\Delta_{i} : \Delta_{i} \cap \Delta_{j}] \geq p_j^{\ell_{j} - n^2}.
$$
Set $\Lambda_i := \cap_{n =1}^i \Delta_{n}$, then
$$
[\Lambda_{j-1} : \Lambda_j] 
= O_j p_j^{\ell_j} [\Delta_j : \Lambda_j]/[\Lambda: \Lambda_{j-1}], 
$$
and $[\Lambda: \Lambda_{j-1}]$ divides $[\Lambda:\Delta_1] \cdots [\Lambda:\Delta_{j-1}]$ for which $p_j^{(j-1)n^2}$ is the largest power of $p_j$ that appears as a factor.
Hence, we obtain
$$
[\Lambda_{j-1} : \Lambda_{j} ] \geq p_j^{\ell_j - (j-1)n^2}.
$$
From this, we obtain the important inequality
\begin{equation} \label{intinequality}
[\Lambda : \Lambda_k] \geq \prod_{j=1}^k p_j^{\ell_j - (j-1)n^2}.
\end{equation}

To employ~(\ref{intinequality}), we need a comparison function. This is where we deviate from the proof in \cite{BM10}.
Define, for $g \in \Lambda \setminus \{1 \}$,
$$
M(g) = \min\{ [\Lambda : \Delta_i ] : g \notin \Delta_i \}.
$$
Let $\hat M$ be the unique continuous extension of $M$ to $\hat \Gamma$.
Then clearly, $D_\Gamma(g) \leq M(g)$, and so
$$
\int \hat D_\Gamma(g) d\mu \leq \int \hat M(g) d\mu.
$$
By studying the partial sums that define $\int \hat M(g) d\mu$, we obtain, for any $n$,
\begin{eqnarray*}
\sum_{k=1}^{n} [\Lambda : \Delta_k] \mu ( \Lambda_{k-1} \setminus \Lambda_k ) &=& \sum_{k=1}^{n} [\Lambda : \Delta_k] \left(\frac{1}{[\Lambda : \Lambda_{k-1}]} - \frac{1}{[\Lambda : \Lambda_k]} \right) \\
&=& \frac{[\Lambda:\Delta_1]}{[\Lambda:\Lambda_0]}- \frac{[\Lambda:\Delta_{n}]}{[\Lambda: \Lambda_{n}]}+\sum_{k=1}^{n-1} \frac{[\Lambda : \Delta_{k+1}] - [\Lambda : \Delta_{k}] }{[\Lambda : \Lambda_{k}]}\\
&<& \frac{[\Lambda:\Delta_1]}{[\Lambda:\Lambda_0]}+\sum_{k=1}^n \frac{[\Lambda : \Delta_{k}]^{1+\delta}}{[\Lambda : \Lambda_{k}]}.
\end{eqnarray*}
The last inequality follows from the conclusion of \cite[Proposition 5.1]{BM10} that $[\Lambda:\Delta_{k+1}] \leq [\Lambda:\Delta_k]^{1+\delta}$. Applying (\ref{Oinequality}) and  (\ref{intinequality}) while plugging in the value for $[\Lambda:\Lambda_k]$ yields

\begin{eqnarray*}
\sum_{k=1}^n \frac{[\Lambda : \Delta_{k}]^{1+\delta}}{[\Lambda : \Lambda_{k}]} &\leq& \sum_{k=1}^n \frac{O_k^{1+\delta} p_k^{(1 + \delta)\ell_k}}{\prod_{j=1}^{k} p_j^{\ell_j - (j-1)n^2}} \\
&\leq& \sum_{k=1}^n \frac{p_k^{(1 + \delta)(n^2+\ell_k)}}{\prod_{j=1}^{k} p_j^{\ell_j - (j-1)n^2}} \\
&=& \sum_{k=1}^n \frac{p_k^{(k + \delta)n^2 +(\delta)\ell_k}}{\prod_{j=1}^{k-1} p_j^{\ell_j - (j-1)n^2}}.
\end{eqnarray*}
We compute the ratio $\left(\text{$k$-th term}\right)/\left(\text{$(k+1)$-th term}\right)$ of the series above:
$$ 
\frac{\frac{p_k^{(k + \delta)n^2 +(\delta)\ell_k}}{\prod_{j=1}^{k-1} p_j^{\ell_j - (j-1)n^2}}}
{\frac{p_{k+1}^{(k + \delta +1)n^2 +(\delta)\ell_{k+1}}}{\prod_{j=1}^{k} p_j^{\ell_j - (j-1)n^2}}}
=
\frac{p_k^{(1+\delta) n^2 + (\delta+1) \ell_k}}
{p_{k+1}^{(k+\delta+1)n^2 + \delta \ell_{k+1}}}
\leq 
\frac{p_k^{(1+\delta) n^2 + (\delta+1) \ell_k}}
{p_{k}^{(k+\delta+1)n^2 + \delta \ell_{k+1}}}
=
p_k^{-kn^2 + \delta (\ell_k -\ell_{k+1}) + \ell_k}
$$
Thus, if $\delta$ is sufficiently small and $k$ is sufficiently large, we have that the exponent above is greater than 1 by (\ref{ellinequality}).
Hence, as $p_k$ is an increasing sequence of integers, the ratio test implies that the resulting series above converges, and $\int \hat M(g) d\mu$ is finite.
We conclude that the normal residual average of $\Lambda$ is finite, as desired.
\end{proof}

\section{Expected depth of random walks on groups}

Fix for the whole section a finitely generated residually finite group $G$ and a finite symmetric generating set $S$.
Consider the simple lazy random walk $(X_n)_{n\geq0}$ on the Cayley graph $\Cay(G,S)$, as defined in \S\ref{sec:RW}. 
%Once the generating set is fixed we will write $X_n$ to simplify the notation. 
Recall that we are interested in 
\begin{equation*}
	\mathbb{E}[D_G(X_n)]
	=\sum_{k\geq 2} k\mathbb{P}[D(X_n) = k]
	=\sum_{k\geq 0} \mathbb{P}[D(X_n)>k]
	=\sum_{k\geq 0} \mathbb{P}(X_n \in \Lambda_k \setminus \{e\}).
\end{equation*}
The second equality is obtained through the same double-sum argument as in the proof of Lemma~\ref{linear-lemma}.

\subsection{First estimates}\label{sec:Fatou}

\begin{prop}\label{prop:Fatou}
	We have 
	\begin{align*}
		\liminf_{n \to \infty} \mathbb E[D(X_n)] \geq 2 + \sum_{k\geq 2} \frac{1}{|G:\Lambda_k|}.
	\end{align*}
\end{prop}

\begin{proof}
	Recall the expression~\eqref{eq:Esum} for $E[D(X_n)]$:
	\begin{align*}
			\mathbb E[D(X_n)] = \sum_{k \geq 0} \mathbb P (X_n \in \Lambda_k\setminus\{e\}).
	\end{align*}
	Also recall from Corollary~\ref{cor:uniform} that 
	\begin{align*}
		\mathbb P (X_n \in \Lambda_k\setminus\{e\}) \xrightarrow[n \to \infty]{} 
		\begin{cases}
		\frac{1}{[G:\Lambda_k]} \quad &\text{ if $k \geq 2$,}\\
		1 \quad &\text{ if $k =0,1$.}
		\end{cases}
	\end{align*}
	The result follows from Fatou's lemma. 
\end{proof}

\begin{cor}
	Suppose $G$ is such that  $\sum_{k\geq 2} \frac{1}{|G:\Lambda_k|}$ diverges. 
	Then $\lim_{n\rightarrow \infty}\mathbb{E}(D(X_n))=\infty$.
	\label{ind}
\end{cor} 

\begin{proof}
	This is a direct consequence of Proposition~\ref{prop:Fatou}.
\end{proof}

\begin{prop}\label{prop:domination}
	Suppose there exists a sequence of positive numbers
	$\{p_k\}_{k\geq 2}$ with
	\begin{itemize}
	\item $\sum_{k\geq 2} p_k<\infty$;
	\item $\mathbb{P}(X_n \in \Lambda_k\setminus \{e\})\leq p_k$
	for all $n\geq 1$ and $k \geq2$.
	\end{itemize}
	Then
	$$\lim_{n\rightarrow \infty}\mathbb{E}[D(X_n)]=2 + \sum_{k\geq 2} \frac{1}{|G:\Lambda_k|} < \infty.$$
	%In particular, $lim_{n\rightarrow \infty}\mathbb{E}(D(Xn^S))$ is independent of the choice of the generating set $S$.
\end{prop}

\begin{proof}
	Fix a sequence $(p_k)_{k\geq 2}$ as above, and set $p_0 = p_1 = 1$. 
	Then the convergence of $\mathbb P (X_n\in~\Lambda_k\setminus\{e\})$ to $\frac{1}{[G:\Lambda_k]}$ is dominated by $p_k$. 
	Since $p_k$ is summable, the dominated convergence theorem implies the desired result. 
\end{proof}

Below, when applying Proposition~\ref{prop:domination}, we will do so using the sequence 
$$ p_k = \sup_{n \geq 0}\mathbb{P}(X_n \in \Lambda_k\setminus \{e\}), \qquad \text{for $k \geq 2$}.$$
This sequence obviously satisfies the domination criterion; one needs to show it is summable in order to apply the proposition. 

\subsection{Sufficient condition using spectral properties}

%
%When writing
%$$\mathbb{E}[D(X_n)] = \sum_{k \geq 0} \mathbb{P}(X_n \in \Lambda_k \setminus \{e\}),$$
%all terms with $k \geq F_G^S(n)$ are null, because $X_n \in B^S_G(n)$ 
%and $\Lambda_{F_G^S(n)} \cap B^S_G(n) = \{e\}$, by definition of $F_G^S(n)$. 
%Thus, we may only sum up to $F_G^S(n)-1$ in the formula above.  
%
%\begin{thm}
%	$$\left|\mathbb{E}[D(X_n)]-2-\sum_{k=2}^{F_G^S(n)-1} \frac{1}{|G:\Lambda_k|}\right| \leq 2F_G^S(n)\mu_{F_G^S(n)}^n+1,$$
%	where $\mu_{F_G^S(n)}$ is the second largest eigenvalue 
%	of the transition matrix of the induced lazy random walk $\tilde{X}_{n}$ on $\Cay(G/\Lambda_{F_G^S(n)}, S)$.
%	%\begin{enumerate}
%	%\item $\xi_n=-\frac{F_G^S(n)}{|G:\Lambda_{F_G^S(n)}|}+2+\sum_{k=2}^{F_G^S(n-1)} \frac{1}{|G:\Lambda_k|}$,
%	%\item $\mu_{F_G^S(n)}$ is the second largest eigenvalue of the transition matrix of the induced lazy random walk $\tilde{X}_{n}$ on $\Cay(G/\Lambda_{F_G^S(n)}, S)$.
%	%\end{enumerate}
%	\label{general}
%\end{thm}
%

\begin{proof}[Proof of Theorem~\ref{kazhdan}]
	Fix a linear group $G$ with Property $(T)$. 
	We will apply Proposition~\ref{prop:domination} to show the desired convergence. 
	Fix some $k \geq 2$ and let us bound $\mathbb{P}(X_n \in \Lambda_k \setminus \{e\})$ for arbitrary $n$. 
	
	First notice that, $X_n\in B_G^S(n)$ and therefore $D_G(X_n) \leq F_G^S(n)$. It follows that 
	\begin{align*}
		\mathbb{P}(X_n \in \Lambda_k \setminus \{e\}) = 0\qquad \text{ if $F_G^S(n) \leq k$}.
	\end{align*}
	Suppose now that $n$ is such that $F_G^S(n) > k$.  
	Recall from Corollary~\ref{cor:uniform} that  
	$|\mathbb{P}(X_n\in \Lambda_k)-\frac{1}{|G:\Lambda_k|}|\leq \mu_k^n$,
	where $\mu_k$ is the second largest eigenvalue of the transition matrix 
	of the induced lazy random walk on $\Cay(G/\Lambda_k, S)$. 
	Now, since $G$ has Property $(T)$, there exists a constant $0<\theta<1$ such that,
	for any normal finite index subgroup $N\lhd G$, 
	the second largest eigenvalue of the Cayley graph of $G/N$ is bounded above by $\theta<1$
	(see \cite{BekHarVal08}; the exact value of $\theta$ does depend on the generating set $S$ of $G$).
	In particular, 
	\begin{align*}
		\mathbb{P}(X_n \in \Lambda_k \setminus \{e\}) 
		\leq \mathbb{P}(X_n \in \Lambda_k ) 
		\leq \frac{1}{|G:\Lambda_k|}+ \mu_k^n 
		\leq \frac{1}{|G:\Lambda_k|}+ \theta^n.
	\end{align*}
	Observe that the right-hand side above is decreasing in $n$, and therefore is maximal when $n$ is minimal. 
	Set $N_k = \inf\{ n \geq 1:\, F_G^S(n) > k\}$. Then, by the above two cases, we deduce that
	\begin{align*}
		\mathbb{P}(X_n \in \Lambda_k \setminus \{e\}) 
		\leq \frac{1}{|G:\Lambda_k|}+ \theta^{N_k} =: p_k \qquad \text{ for all $n \geq 1$ and $k \geq 2$}.
	\end{align*}
	The values $(p_k)_{k\geq2}$ defined above satisfy the second property of Proposition~\ref{prop:domination}; 
	we will show now that they also satisfy the first.

	By \cite{BM13}, there exist $b \in \mathbb N$ and $C > 0$ such that $F_G^S(n)\leq Cn^{b}$ for all $n \geq 1$. 
	In particular, for any $k \geq 2$, %$CN_k^b \geq F_G^S(n) > k$ so 
	$N_k \geq C' k^{1/b}$ for some constant $C'>0$ that does not depend on $k$. 
	Moreover, $\sum_{k\geq 2} \frac{1}{|G:\Lambda_k|}$ is finite by Lemma~\ref{linear-lemma} and Theorem~\ref{thm:linear-finite average}.
	Thus
	\begin{align*}
		\sum_{k\geq2}p_k \leq \sum_{k\geq 2} \frac{1}{|G:\Lambda_k|} + \sum_{k\geq2}  \theta^{C' k^{1/b}} <\infty.
	\end{align*}
	Applying Proposition~\ref{prop:domination} yields the desired result. 
\end{proof}	

\subsection{Sufficient condition: abelian groups}
  
\begin{lem}\label{lem:RWonZ}
	Let $(X_n)_{n \geq 0}$ be a lazy random walk on $\ZZ$ (that is on a Cayley graph of $\ZZ$, as in \S\ref{sec:RW}). 
	Then there exists a constant $C > 0$ such that, 
	for all $m \geq 1$
	\begin{align}\label{eq:RWonZ1}
		\sup_{n \geq 0} \PP\big(X_n \in m\ZZ\setminus \{0\}\big) \leq \frac{C}{\sqrt m}.
	\end{align}
	In particular, there exists $c > 0$ such that, for $k \geq 2$, 
	\begin{align}\label{eq:RWonZ2}
		\sup_{n \geq0} \PP\big(X_n \in \Lambda_k(\ZZ)\setminus \{0\}\big) \leq e^{-ck}.
	\end{align}
\end{lem}

\begin{proof}
	We start with the proof of~\eqref{eq:RWonZ1}. Let $c_0 >0$ be such that $|X_1| < \frac1{2c_0}$ almost surely. 
	Below, write $c_0 m $ for the integer part of $c_0m$ so as not to over-burden notation. 
	Then, for $n \leq c_0 m$, $\PP(X_n \in m\ZZ\setminus \{0\}) = 0$. 
	For $n \geq c_0 m$, write
	\begin{align}\label{eq:RWonZ3}
		\PP(X_n \in m\ZZ\setminus \{0\}) 
		%= \EE\big[\PP(X_n \in m\ZZ\setminus \{0\})\big| X_{n - c_0m/2}\big]
		= \sum_{\ell \in \ZZ} \PP\big(X_n \in m\ZZ\setminus \{0\}\big| X_{n - c_0m} = \ell\big) 
		\PP(X_{n - c_0m} = \ell).
	\end{align}
	Now notice that, due to the choice of $c_0$, $|X_n - X_{n - c_0m}| <m/2$ almost surely.
	However, for any fixed $\ell \in \ZZ$, 
	there exists at most one element $m(\ell) \in m\ZZ$ with $|\ell - m(\ell)| < m/2$.
	If no such element exists, choose $m(\ell) \in m\ZZ$ arbitrarily. 
	Thus
	\begin{align*}
		\PP\big(X_n \in m\ZZ\setminus \{0\}\big| X_{n - c_0m} = \ell\big) 
		& \leq\PP\big(X_n = m(\ell)\big| X_{n - c_0m} = \ell\big) \\
		& =\PP\big(X_{c_0m} = m(\ell) - \ell\big) 
		% \leq \sup_{k \in \ZZ}\PP\big(X_{c_0m} = k\big) 
		\leq \frac{C}{\sqrt{c_0 m}},
	\end{align*}
	where the last inequality is due to \cite[Thm. VI.5.1]{Varopoulos_book} and $C > 0$ 
	is some fixed constant depending only on the transition probability of the random walk. 
	When injecting the above in~\eqref{eq:RWonZ3}, we find
	\begin{align*}
		\PP(X_n \in m\ZZ\setminus \{0\}) 
			\leq \sum_{\ell \in \ZZ}  \frac{C}{\sqrt{c_0 m}}\PP(X_{n - c_0m} = \ell)
			= \frac{C}{\sqrt{c_0 m}}.
	\end{align*}
	Since the right-hand side does not depend on $n$, this implies~\eqref{eq:RWonZ1} with an adjusted value of~$C$. 
	\smallskip 
	
	We move on to proving~\eqref{eq:RWonZ2}. The (normal) subgroups of $\ZZ$ are of the form $k\ZZ$, with $k$ being their index. 
	Thus, for $k\geq 2$, 
	$$ \Lambda_k(\ZZ) = m_k \ZZ,$$
	where $m_k$ is the least common multiple of $1,\dots, k$ (see \cite{BBKM13}). 
	It follows from the Prime Number Theorem that there exists a constant $c > 0$ such that 
	$$ m_k \geq \exp(ck),\qquad \forall k\geq 2.$$
	The above bound, together with~\eqref{eq:RWonZ1}, implies~\eqref{eq:RWonZ2} with an adjusted value of $c$. 
\end{proof}

\begin{cor}[Expected depth for $\Z$.]
	Let  $(X_n)_{n\geq 0}$ be a lazy random walk on $\Z$ (that is, on a Cayley graph of $\Z$, as in \S\ref{sec:RW}). Then 
	\begin{equation}
		\E[D_{\Z}(X_n)]\xrightarrow[n\to \infty]{} 2+\sum_{k\geq 2}\frac{1}{|\Z:\Lambda_k|}<\infty.
	\end{equation}
	\label{Z}
\end{cor}

\begin{proof}
	For $k \geq 2$, set $p_k = \sup_{n \geq0} \PP\big(X_n \in \Lambda_k(\ZZ)\setminus \{0\}\big)$.
	By Lemma~\ref{lem:RWonZ}, $\sum_{k \geq 2} p_k <\infty$. 
	
	Moreover, the sequence $(p_k)_{k \geq2}$ dominates the convergence of 
	$\PP\big(X_n \in \Lambda_k(\ZZ)\setminus \{0\}\big)$, 
	as required in Proposition~\ref{prop:domination}. The conclusion follows. 
\end{proof}

\begin{prop}\label{prop:direct_product}
	Let $G$ and $H$ be two finitely generated residually finite groups. 
	Let  $(X_n)_{n\geq 0}$ be a random walk on a Cayley graph of $G\times  H$, as in \S\ref{sec:RW}. 
	Then 
	\begin{align*}
		&\PP[D_{G\times H} (X_n)>k]\leq \PP[D_G(Y_n)>k] + \PP[D_H(Z_n)>k] \qquad \text{for all $k \geq 0$,} 
%		\text{and }\\
%		&\mathbb{E}[D_{G\times H}(X_n)]\leq \E[D_G(Y_n)]+\E[D_H(Z_n)],
	\end{align*}
	where $(Y_n)_{n\geq0}$ and $(Z_n)_{n\geq0}$ are the random walks on $G$ and $H$, respectively, induced by $(X_{n})_{n \geq 0}$.
\end{prop}

\begin{proof}
%Recall that $\E[D_{G\times H}(X_n)]=\sum_{k\geq 0} \PP[D_{G\times H}(X_n)>k]$. 
Notice that for $g=(g_1, g_2) \in G\times H$  we have estimates
\begin{align*}
&D_{G\times H}(g)\leq D_G(g_1) \text{ if } g_1\neq e\\
&D_{G\times H}(g)\leq D_G(g_2) \text{ if } g_2\neq e
\end{align*}

Therefore $\PP[D_{G\times H} (X_n)>k]\leq \PP[D_G(Y_n)>k] + \PP[D_H(Z_n)>k]$.
\end{proof}

\begin{cor} 
	Let $G$ be a finitely generated abelian group and let  $(X_n)_{n\geq 0}$ 
	be a lazy random walk on its Cayley graph, as in \S\ref{sec:RW}. Then
	\begin{equation*}
	\E[D_G(X_n)]\xrightarrow[n\rightarrow \infty]{} 2+\sum_{k\geq 2} \frac{1}{|G:\Lambda_k|}<\infty,
	\end{equation*}
	for any finite generating set $S$ of $G$.
\end{cor}

\begin{proof}
	Let $G$ be a finitely generated abelian group. 
	Then it may be written as $G = \ZZ \times \dots \times \ZZ \times H$, where the product contains $j$ copies of $\ZZ$ 
	and $H$ is a finite abelian group. 
	The depth of elements of $H$ is bounded by $|H|$. 
	By Proposition~\ref{prop:direct_product} and Lemma~\ref{lem:RWonZ}, there exists $c > 0$ such that 
	$$ \PP\big(X_n \in \Lambda_k(G)\setminus \{0\}\big) \leq j e^{-ck}, \qquad \forall k > |H|.$$
	We conclude using the same domination argument as in the proof of Corollary~\ref{Z}.
\end{proof}
\smallskip 

\subsection{Nilpotent groups: proof of Theorem~\ref{nilp}}
~\smallskip

\noindent \textbf{Torsion free case.} Suppose first that $G$ is a finitely generated and torsion free nilpotent group, different from $\ZZ$. Observe that $G$ is a poly-$C_{\infty}$ group, and in particular, $G = \ZZ \ltimes H$ where $H$ is a non-trivial finitely generated nilpotent group. 
	Let $S$ be a system of generators of $G$ and consider the lazy random walk on $G$ with steps taken uniformly in $S$. 
	We will write it in the product form $(X_n,Y_n)_{n \geq 0}$, where $X_n \in \ZZ$ and $Y_n \in H$, for all $n \geq 0$. 
	Notice then that $(X_n)_{n \geq 0}$ is a lazy random walk on $\ZZ$, as treated in Lemma~\ref{lem:RWonZ}. 
	This is not true on the second coordinate: $(Y_n)_{n\geq0}$ is not a random walk on $H$, it is not even a Markov process.
	
	For $ k \geq 2$, let 
	\begin{align}
		p_k(G) = \sup_{n \geq 0}  \PP \big[D_G(X_n,Y_n) > k \big], 
	\end{align}
	and recall from Proposition~\ref{prop:domination} (and the commentary below it) that our goal is to prove that $\sum_k p_k <\infty$.  
	
	One may easily check that, for any $m \in \NN$, $m\ZZ \ltimes H$ is a normal subgroup of $G$. 
	Thus, for all $x \in \ZZ\setminus \{0\}$ and $y \in H$, 
	$$ D_G(x,y) \leq D_{\ZZ}(x).$$
	We may therefore bound $p_k$ by 
	\begin{align*}
		p_k(G)
		&\leq \sup_{n \geq 0} \PP \big[D_G(X_n) > k \big] + \PP \big[X_n = 0 \text{ and } D_G(0,Y_n) > k \big]\\
		&\leq p_k(\ZZ) + \sup_{n \geq0}\PP \big[X_n = 0 \text{ and } D_G(0,Y_n) > k \big].
	\end{align*}
	In the above, $p_k(\ZZ)= \sup_{n \geq 0}  \PP \big[D_{\Z}(X_n) > k \big]$. 
	We have shown in Lemma~\ref{lem:RWonZ} that $\sum_k p_k(\Z) < \infty$, 
	and we may focus on whether the second supremum is summable. 
	
	Fix $k \geq 2$. 
	Since $G$ is nilpotent, there exists $c > 0$ such that 
	$$F_G^S(n) \leq c (\log n )^{h(G)} ,$$
	where $h(G)$ is the Hirsch length of $G$ (see Theorem~\ref{rf-nilp}).
	This should be understood as follows. In order for an element $g \in G$ to have $D_G(g) \geq k$, 
	it is necessary that $\| g\|_S \geq \exp(C k^{1 / h(G)})$, where $\|g\|_S$ denotes the length of $g$ with respect to the generating set $S$ and $C > 0$ is a constant independent of $g$.
	
	In particular, we conclude that 
	\begin{align*}
		\PP\big[X_n = 0 \text{ and } D_G(0,Y_n) > k \big] =0 \quad 
		& \text{	if $n < \exp(C k^{1 / h(G)})$;}\\
		\PP\big[X_n = 0 \text{ and } D_G(0,Y_n) > k \big]\leq \PP(X_n = 0)\quad 
		& \text{	if $n \geq \exp(C k^{1 / h(G)})$.}	
	\end{align*}
	In treating the second case, 
	observe that, since $(X_n)_{n\geq 0}$ is a random walk on $\ZZ$,
	$$\PP(X_n = 0) \leq c_0 n^{-1/2}, $$
	for some constant $c_0 > 0$ (see \cite[Thm. VI.5.1]{Varopoulos_book}).
	Thus
	\begin{align*}
		\PP\big[X_n = 0 \text{ and } D_G(0,Y_n) > k \big] 
		\leq c_0 \exp\Big(- \frac{C}2 k^{\frac{1}{ h(G)} }\Big),\qquad \forall n \in \NN.
	\end{align*}
	We conclude that 
	\begin{align*}
		\sum_{k \geq2} \sup_{n \geq0}\PP \big[X_n = 0 \text{ and } D_G(0,Y_n) > k \big] < \infty, 
	\end{align*}
	and therefore that $\sum_{k \geq 0} p_k(G) < \infty$. 
	\medskip
	
\noindent \textbf{General nilpotent case.} 
Let now $G$ be a finitely generated nilpotent group. Consider the set $T(G)$ of all torsion elements of $G$. Since $G$ is nilpotent, the set $T(G)$ is a finite normal subgroup in $G$. Consider an epimorphism $\pi\colon G\rightarrow G/T(G)$. Denote by $H$ the quotient $G/T(G)$ and notice that $H$ is a finitely generated torsion-free nilpotent group. 

Observe that for any non-trivial element in $H$ detected by a normal subgroup in $H$ of index $k$ there exists a normal subgroup in $G$ of index at most $k$ that detects its preimage. In other words, for all $g \in G \setminus T(G)$, 
	\begin{align*}
	 	D_G(g) \leq D_H(\pi(g)).
	\end{align*}

The random walk $(X_n)_{n \geq 0}$ on $G$ induces a random walk $(\pi(X_n))_{n \geq 0}$ on $H$. 
Let $d = \max \{D_G(g) , g\in T(G)\}$. 
Due to the observation above,
for all $k>d$,
$$\PP[D_G(X_n)\geq k] \leq \PP[D_H(\pi(X_n))\geq k].$$
We deduce from the case of torsion free nilpotent groups that $\sum_{k>d}\sup_{n\geq 0} \PP[D_G(X_n)\geq k]<\infty$. 
Then the second point of Proposition~\ref{prop:domination} applies and we obtain the desired conclusion. 
\hfill $\square$

\subsection{A counter example}\label{sec:counter_ex}

%\begin{prob}
%Find a finitely generated group $G$ for which $\lim_{n\rightarrow \infty}\mathbb{E}(D(X_n))=\infty$.
%\end{prob}

\begin{prop}[Groups with infinite expected depth]\label{prop:counter_ex}
	There exists a finitely generated residually finite group $G$ such that 
	$$\lim_{n\rightarrow \infty}\mathbb{E}(D(X_n))=\infty,$$
	but for which the ``presumed'' limit $\sum_{k\geq 2} \frac{1}{|G:\Lambda_k|}$ is finite.
\end{prop}

\begin{proof}
The existence of finitely generated residually finite groups with arbitrary large residual finiteness growth was shown in \cite{BS13b}. 

Let $H$ be a two-generated group (with generators $a,b$) such that, 
for any $n\geq 8$, there exists an element $h_n$ in the ball of radius $n$ of $H$ with $D_H(h_n)\geq 24^n$. 
Let  $(X_n)_{n\geq 0}$ be a lazy simple random walk on the Cayley graph of $G=H\times \Z$ with the natural choice of $3$ generators (that is $(a,0)$, $(b,0)$ and $(e_H,1)$, where $a$ and $b$ are the two generators of $H$ mentioned above) and their inverses. 

Then, for any $n \geq 1$, $\mathbb{P}[X_n=(h_n,0)]\geq 12^{-n}$. Therefore 
$$
	\E[D_G(X_n)]\geq \p[X_n=(h_n,0)]\cdot D_{G}[(h_n,0)] 
	=  \p[X_n=(h_n,0)]\cdot D_{H}(h_n)\geq 2^n,
$$ 
Hence the expectation of the depth of $X_n$ tends to infinity.

Furthermore, observe that $\Lambda_k(G)$ is a subgroup of $H\times \Lambda_k(\Z)$ and hence 
\begin{equation*}
|G:\Lambda_k(G)|\geq |\Z: \Lambda_k(\Z)|.
\end{equation*}
 It follows that $\sum_{k\geq 2} \frac{1}{|G:\Lambda_k(G)|}\leq \sum_{k\geq 2}\frac{1}{|\Z:\Lambda_k(\Z)|}<\infty$. 
\end{proof}

%\section{Expected depth for locally compact groups}
%The expected depth can be defined in a similar way for a locally compact group $G$ generated by a compact set.

%\begin{defn} \textcolor{red}{Definition for expected depth for locally compact groups using Haar measure}
%\end{defn}

%\begin{prop}
%\label{SL3Zp}
%The expected depth of $SL_3(\Z_p)$ diverges.
%\end{prop}
%\begin{proof}\textcolor{red}{Khalid's proof}

%\end{proof}

\bibliography{References}

\def\cprime{$'$} \def\cprime{$'$} \def\cprime{$'$} \def\cprime{$'$}
  \def\cprime{$'$} \def\cprime{$'$} \def\cprime{$'$} \def\cprime{$'$}
  \def\cprime{$'$} \def\cprime{$'$}
\begin{thebibliography}{BdlHV08}

\bibitem[BBRKM]{BBKM13}
I.~Biringer, K.~Bou-Rabee, M.~Kassabov, and F.~Matucci.
\newblock Intersection growth in nilpotent groups.
\newblock submitted, arXiv:math.GR/1309.7993.

\bibitem[BdlHV08]{BekHarVal08}
B.~Bekka, P.~de~la Harpe, and A.~Valette.
\newblock {\em Kazhdan's property~({T})}, volume~11 of {\em New Mathematical
  Monographs}.
\newblock Cambridge University Press, Cambridge, 2008.

\bibitem[BR10]{B09}
K.~Bou-Rabee.
\newblock Quantifying residual finiteness.
\newblock {\em J. Algebra}, 323(3):729--737, 2010.

\bibitem[BRK12]{BK12}
K.~Bou-Rabee and T.~Kaletha.
\newblock Quantifying residual finiteness of arithmetic groups.
\newblock {\em Compos. Math.}, 148(3):907--920, 2012.

\bibitem[BRM10]{BM10}
K.~Bou-Rabee and D.~B. McReynolds.
\newblock Bertrand's postulate and subgroup growth.
\newblock {\em J. Algebra}, 324(4):793--819, 2010.

\bibitem[BRM15]{BM13}
K.~Bou-Rabee and D.~B. McReynolds.
\newblock Extremal behavior of divisibility functions.
\newblock {\em Geom. Dedicata}, 175:407--415, 2015.

\bibitem[BRS16]{BS13b}
K.~Bou-Rabee and B.~Seward.
\newblock Arbitrarily large residual finiteness growth.
\newblock {\em J. Reine Angew. Math.}, 710:199--204, 2016.

\bibitem[BV02]{BV02}
J.~Burillo and E.~Ventura.
\newblock Counting primitive elements in free groups.
\newblock {\em Geom. Dedicata}, 93:143--162, 2002.

\bibitem[Hal69]{MR0283083}
P.~Hall.
\newblock {\em The {E}dmonton notes on nilpotent groups}.
\newblock Queen Mary College Mathematics Notes. Mathematics Department, Queen
  Mary College, London, 1969.

\bibitem[Pan08]{Pansu08}
P.~Pansu.
\newblock Croissance des boules et des g{\'e}od{\'e}siques ferm{\'e}es dans les
  nilvari{\'e}t{\'e}s.
\newblock {\em Ergodic Theory and Dynamical Systems}, 3(3):415--445, Sep 2008.

\bibitem[Seg83]{MR713786}
D.~Segal.
\newblock {\em Polycyclic groups}, volume~82 of {\em Cambridge Tracts in
  Mathematics}.
\newblock Cambridge University Press, Cambridge, 1983.

\bibitem[VSCC93]{Varopoulos_book}
N.~T. Varopoulos, L.~Saloff-Coste, and T.~Coulhon.
\newblock {\em Analysis and Geometry on Groups:}.
\newblock Cambridge Tracts in Mathematics. Cambridge University Press, Jan
  1993.

\bibitem[Wil98]{MR1691054}
J.~S. Wilson.
\newblock {\em Profinite groups}, volume~19 of {\em London Mathematical Society
  Monographs. New Series}.
\newblock The Clarendon Press, Oxford University Press, New York, 1998.

\end{thebibliography}
\bibliographystyle{alpha}
\end{document}